\documentclass[11pt]{article}
\usepackage[T1]{fontenc}
\usepackage{a4wide}
\usepackage{lipsum}

\usepackage{graphicx}
\usepackage{amsmath,amsfonts,amsthm}
\usepackage{hyperref}
\usepackage{mathtools}
\usepackage{amssymb}
\usepackage[english]{babel}  
\usepackage{url}
\usepackage[title]{appendix}%
\usepackage[justification=centering]{caption}
\usepackage{subcaption}

\def\Tau{\mathcal{T}}
\def\F{\mathcal{F}}
\DeclareMathOperator{\Gau}{Gau}
\DeclareMathOperator{\Tait}{Tai}

\newtheorem{lemma}{Lemma}
\newtheorem{theorem}{Theorem}
\newtheorem{remark}{Remark}
\newtheorem{proposition}{Proposition}
\newtheorem{conclusion}{Corollary}

\date{}
\author{
Ilyas Kalimullin 
\and
Eduard Lerner\thanks{Supported by RSF (project No. 24-21-00158).}
}
\title{The $\alpha$-representation for Tait coloring and sums over spanning trees}

\begin{document}

\maketitle

\begin{abstract}
Consider a connected pseudograph $H$ such that each edge is associated with weight $x_e$, 
$x_e\in\mathbb F_3$; $\Tau(H)$ is the set of spanning trees of graph $H$. Assume that
$s(H;{\mathbf x})=\sum_{T\in\Tau(H)} \prod_{e\in E(T)} x_e$. Let $G$ be a maximal planar graph (arbitrary planar triangulation) such that each face $F$ is assigned the value $\alpha(F)=\pm1 \in \mathbb{F}_3$. Then we can associate each edge with $x_e=\alpha(F'_e)+\alpha(F''_e)$, where $F'_e$ and $F''_e$ are the faces containing edge~$e$. Let us define the value $w_G({\mathbf x})$ as $\left(\frac{s(G/W^*({\mathbf x});\,{\mathbf x})}3\right)/(-3)^{\left(|V(G/W^*({\mathbf x}))| - 1\right)/2}$; here $\left(\frac{x}3\right)$ is the Legendre symbol, $G/W$ is the graph with the contracted set of vertices~$W$, while $W^*({\mathbf x})$ is a set of vertices~$W$, $W \subseteq V(G)$, with minimal cardinality such that $s(G/W;{\mathbf x})$ differs from zero. In the following, we prove that the number of Tait colorings for graph~$G$ equals the tripled sum $w_G({\mathbf x}(\alpha))$ with respect to all possible vectors $\alpha \in \{-1, 1\}^{\F(G)}$ such that $G/W^*({\mathbf x}(\alpha))$ has an odd number of vertices, where $\F(G)$ is the set of faces of graph $G$.
\end{abstract}

\section{Introduction}
The idea of this work has a long history.
Let notation $\Tau(G)$ stand for the set of spanning trees of connected graph~$G$. Consider sums
\begin{equation} \label{eq:2}
s(G;{\mathbf x})=\sum_{T \in \Tau(G)}
  {\prod_{e\in E(T)} {x_e } },
\end{equation}
where $x_e$ are elements of finite field~${\mathbb F}_q$.
In December 1997, when giving a talk at the Gelfand Seminar at Rutgers University,
Maxim Kontsevich proposed the conjecture that the number of non-zero values of~\eqref{eq:2}
for ${\mathbf x}  \in \mathbb{F}_q^{E(G)}$ is a polynomial with respect to~$q$.
This conjecture was inspired by studying analogous sums (with real positive $x_e$) in quantum field theory. Although this conjecture was never published,
it has aroused the interest of experts in combinatorics (see~\cite{stanleyArticle,chung,Stembridge}).
Sometime later, this conjecture was refuted~\cite{belk}.
Note that in the refuted conjecture one actually considers the sum of weights $w_G({\mathbf x})$ that equal one when $s(G;{\mathbf x})$ differs from zero. A proper weight $w_G({\mathbf x})$ is related to the value of a multidimensional Gaussian sum (an analog of the Gaussian integral) over a finite field such that the Laplace--Kirchhoff matrix of quadratic form is parameterized by values~${\mathbf x}$. For example, in the case of field~$\mathbb F_3$, an analog of Gaussian integral obeys formula~\eqref{Gau3} given below.
If we apply these formulas and carefully adjust the techniques of the so-called $\alpha$-representation, which are used in quantum theory in the case of a real field, to the case of a finite one, then we obtain a new representation for the flow polynomial of graph~\cite{EJC}. Moreover, this representation allows for a generalization for the case of an arbitrary matroid representable over field~${\mathbb F}_q$~\cite{arXiv}. 
In the case of a regular matroid, this formula is even simpler.

The goal of this paper is to obtain the $\alpha$-representation for the number of Tait colorings for an arbitrary maximal planar graph. For the cubic graph dual to the considered one, this representation was recently obtained in~\cite{arXiv}. In our case, there occur sums with respect to spanning trees, which unites this representation with the initial Kontsevich conjecture.

Let us now state the main result. We will consider not only the sums~\eqref{eq:2} for initial simple graphs~$G$, but also sums~$s(H;{\mathbf x})$ for pseudographs~$H$ obtained from graph~$G$ by contracting all the vertices that belong to set~$W$,
$W\subseteq V(G)$; denote the pseudograph $H$ as $G/W$. 
Evidently, the value $s(H;{\mathbf x})$ is independent of loops of graph~$H$ (as distinct from multiple edges).
Let notation $W^*({\mathbf x})$ stand for an arbitrary set of vertices $W$ with minimal cardinality such that sum $s(G/W;{\mathbf x})$ differs from zero.

In what follows, we consider only the field $\mathbb F_3$ of three elements $\{-1,0,1\}$. In such notation of elements, the {\it Legendre symbol}
$\left(\frac{x}3\right)$, $x\in\mathbb F_3$, coincides with the corresponding real value~$x$. Assume that for arbitrary ${\mathbf x}\in 
\mathbb{F}_3^{E(G)}$ the weight $w_G({\mathbf x})$ obeys the formula 
\begin{equation}
\label{Gau3}
w_G({\mathbf x})=\left(\frac{s(G/W^*({\mathbf x});{\mathbf x})}3\right)/(-3)^{\left(|V(G/W^*({\mathbf x}))| - 1\right)/2}.
\end{equation}
Here $\sqrt{-3}= - i\sqrt{3}$, although this fact does not affect the statement of the main theorem, because it contains only integer powers of $(-3)$.
In the following, we prove that the weight $w_G({\mathbf x})$ is independent of the choice of sets $W^*({\mathbf x})$ (with equal cardinalities).

Let $G$ be {\it a maximal planar graph}, i.e., a planar graph such that each face is a triangle. Recall that {\it a Tait coloring} is a coloring of all edges of graph~$G$ in three colors so that edges of one and the same face are colored differently. The existence of such a coloring for any~$G$ is equivalent to the assertion of the Four Color Theorem.
Denote {\it the number of Tait colorings} for graph~$G$ as $\Tait(G)$. Evidently, $\Tait(G)$ is a multiple of 3. We need the value $\Tait_0(G)=\Tait(G)/3$.

Let notation $\F(G)$ stand for {\it the set of faces} of graph~$G$. 
Let us associate each face $F$ with the variable $\alpha(F)$ which takes on values in the set $\{1,-1\}$ of invertible elements of field $\mathbb F_3$ 
({\it below we denote this set as $\mathbb F^*_3$}). The vector $(\alpha(F), F\in \F(G))$ corresponds to
${\mathbf x}(\alpha)=(x_e, e\in E(G))$, where $x_e=\alpha(F'_e)+\alpha(F''_e)$; here
$F'_e$ and $F''_e$ are faces containing edge~$e$. 

\begin{theorem}
\label{th:main}
The following formula is valid:
$\Tait_0(G)= \sum w(G;{\mathbf x}(\alpha))$;
the sum is calculated with respect to all vectors $\alpha\in\left(\mathbb F_3^*\right)^{\F(G)}$ such that $G/W^*({\mathbf x}(\alpha))$ has an odd number of vertices.
\end{theorem}

\begin{figure}[ht]
    \centering
    \begin{subfigure}{.3\textwidth}
        \centering
        \includegraphics[width=.9\linewidth]{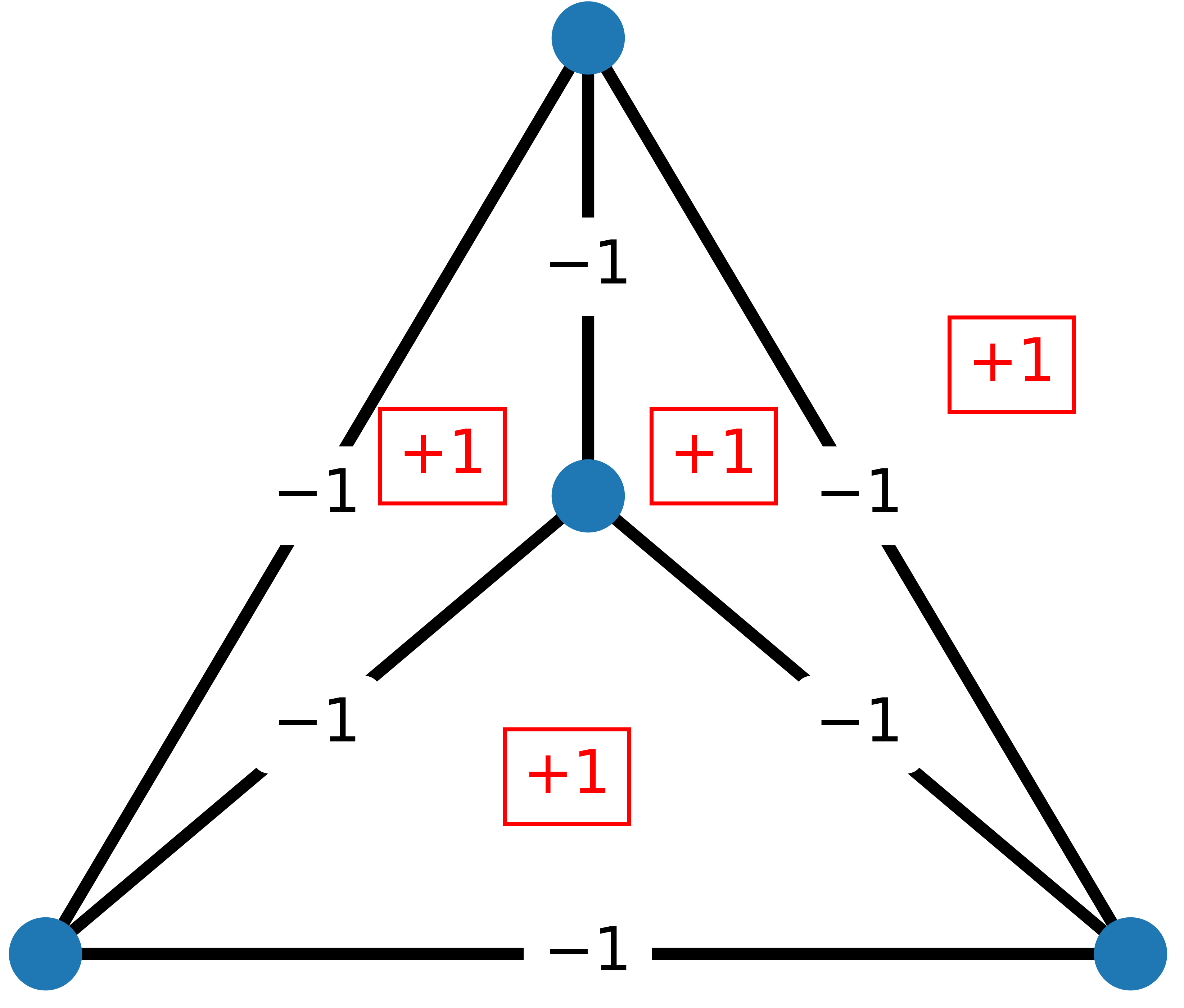}
        \caption{Case 1 of vector $\alpha$ for $K_4$}
        \label{fig:k_4_case_1}
    \end{subfigure}
    \begin{subfigure}{.3\textwidth}
       \centering
        \includegraphics[width=.9\linewidth]{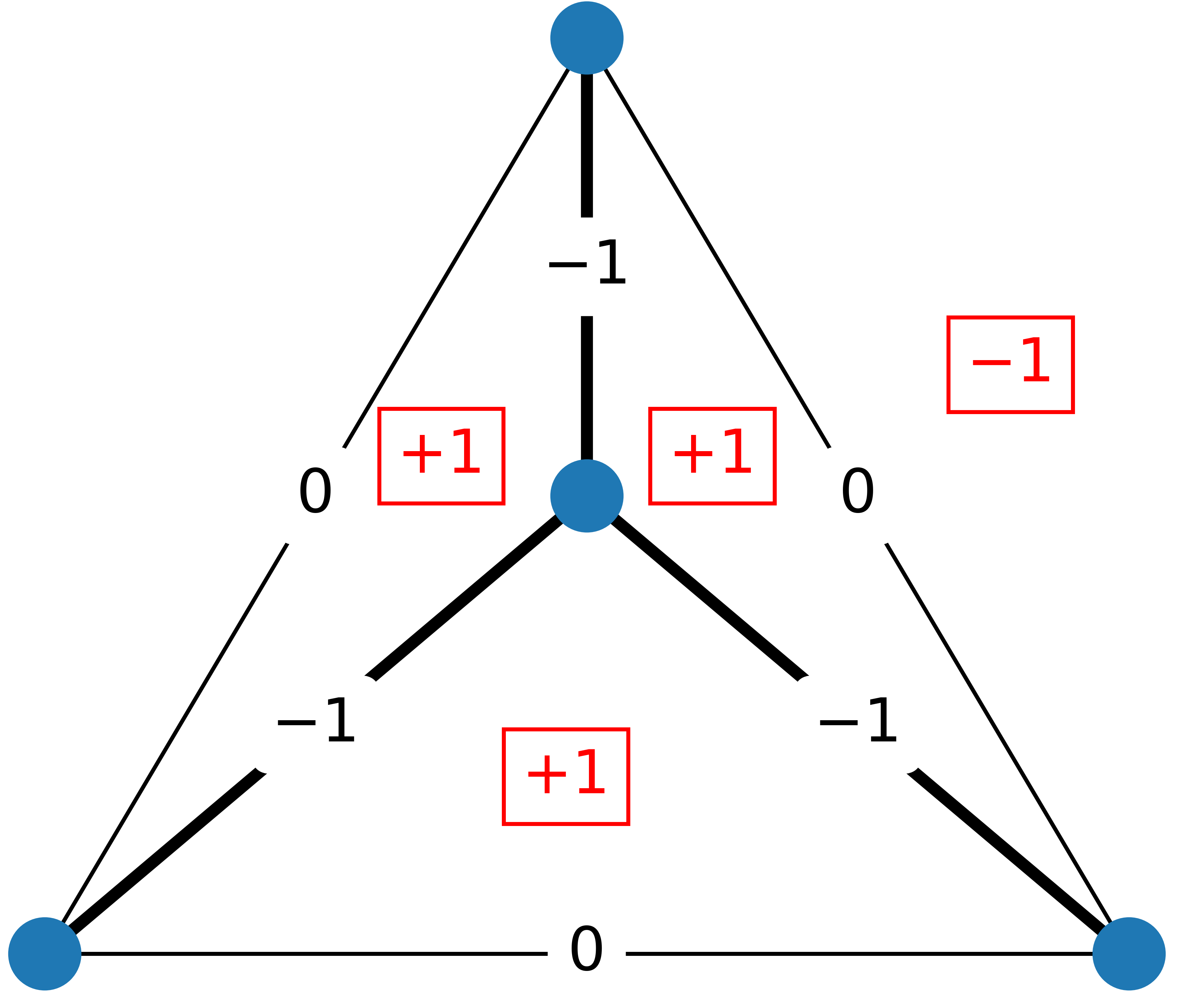}
        \caption{Case 2 of vector $\alpha$ for $K_4$}
        \label{fig:k_4_case_2}
    \end{subfigure}
    \begin{subfigure}{.3\textwidth}
        \centering
        \includegraphics[width=.9\linewidth]{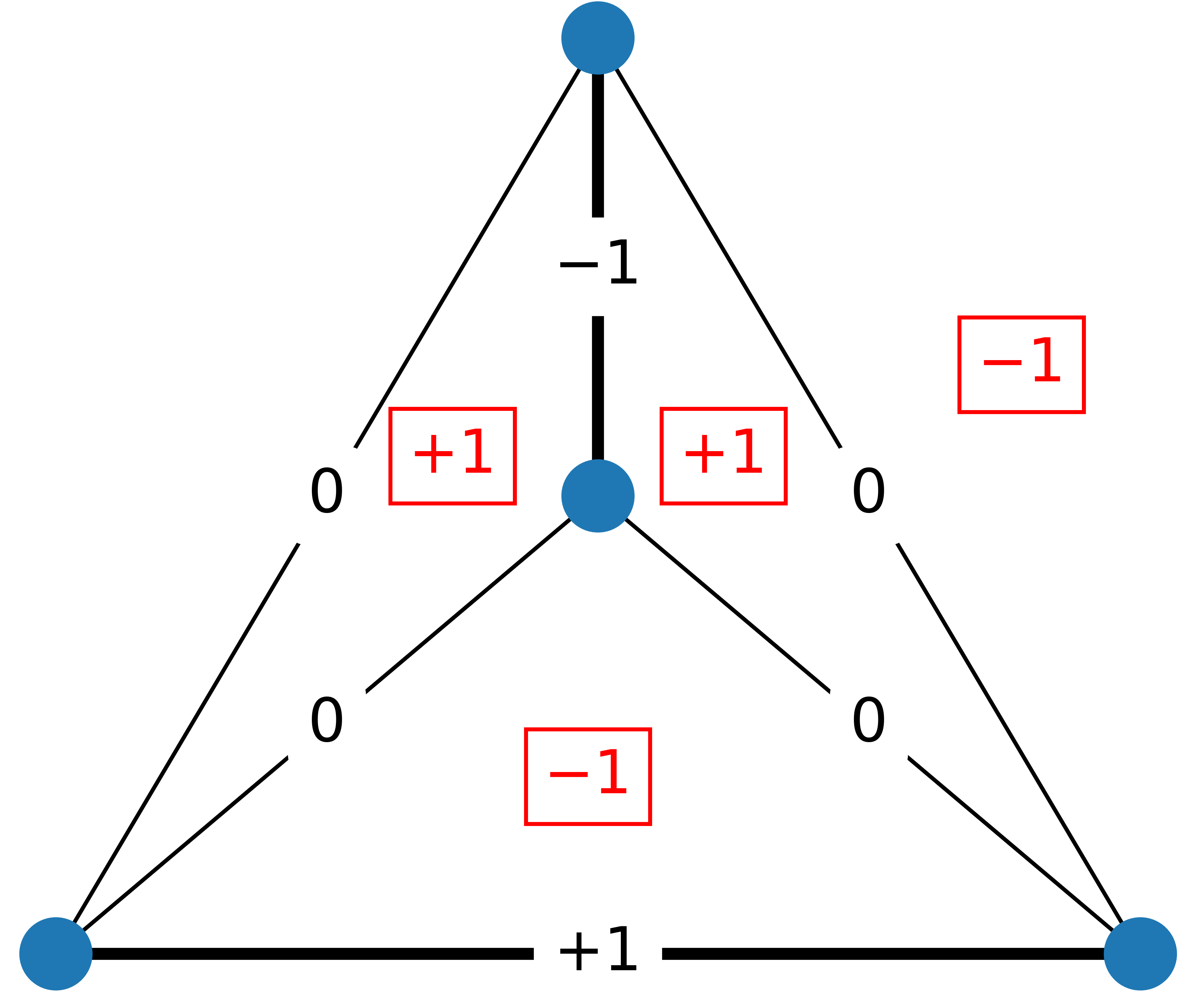}
        \caption{Case 3 of vector $\alpha$ for $K_4$}
        \label{fig:k_4_case_3}
    \end{subfigure}
    \caption{Three cases of vector $\alpha$ for $K_4$}
    \label{fig:k_4}
\end{figure}

For example, in the case of $K_4$, there are 16 vectors $\alpha \in \left(\mathbb{F}_3^*\right)^{\F(K_4)}$ that fall into three cases:

\begin{enumerate}
    \item If $\alpha$ is the same for all faces (see Fig.~\ref{fig:k_4_case_1}), e.g. if $\alpha=(1,1,1,1)$, then for every edge~$e$, $x_e=-1$. Thus, for every spanning tree $T \in \Tau(K_4)$, $\prod_{e\in T}x_e=-1$, and $s(K_4;\mathbf{x}(\alpha))=-1\cdot16=-1$. Therefore, a vertex contraction is not needed: $W^*=\emptyset$ and $K_4/W^*=K_4$ which has an even number of vertices, meaning that it does not change the final sum.
    \item If $\alpha$ differs for only one face (see Fig.~\ref{fig:k_4_case_2}), then three out of six edges have values $x_e=0$ and there is exactly one spanning tree $T'\in\Tau(K_4)$ such that $\prod_{e\in T'}x_e\neq0$, meaning that $s(K_4;\mathbf{x}(\alpha))=\prod_{e\in T'}x_e\neq0$. Once again, no vertex contraction is needed, and values $w_{K_4}({\mathbf x}(\alpha))$ and $w_{K_4}({\mathbf x}(-\alpha))$ cancel each other out.
    \item Finally, if two faces have the value $+1$ and the other two have $-1$ (see Fig.~\ref{fig:k_4_case_3}), then $s(K_4;\mathbf{x}(\alpha))=0$ and two vertices contraction is needed. In this case there is exactly one spanning tree $T'\in\Tau(K_4/W^*)$ such that a product $\prod_{e\in T'}x_e$ is non-zero. It will always have value $-1 \cdot1=-1$, meaning that $s(K_4/W^*;\mathbf{x}(\alpha))=-1$. Therefore, $w_{K_4}(\mathbf{x}(\alpha))= \left( -1 \right) / (-3)=\frac 1 3$. There are exactly ${4 \choose 2}=6$ values of $\alpha$ that fall to this case, thus $\Tait_0(K_4) = 6 \cdot \frac 1 3 = 2$.
\end{enumerate}

Note that the Four Color Theorem allows for an elegant algebraic statement in terms of the {\it graph polynomial} of~$G$ (this notion was introduced by N.Alon and M.~Tarsi in~\cite{alon1}, though the main variant of Theorem~1.1 in this paper was proposed earlier by Yu.~V.~ Matiyasevich in~\cite{matDA}). In particular $\Tait(G)$ coincides with certain coefficients of the graph polynomial of the line graph of~$\widetilde G$, where $\widetilde G$ is the dual graph to~$G$ (see~\cite{sheim}, \cite{alon1}, \cite{POMI}). See related bibliographic references in~\cite{lastMatiyas}.

\section{Gaussian sums and the Laplace--Kirchhoff matrix}

In the case of a real field, the application of the classical $\alpha$-representation implies the use of explicit formulas for the calculation of Gaussian integrals with the imaginary unit in the exponent. In the case of finite fields, we use multidimensional Gaussian sums. The statement of Theorem~\ref{th:main} contains explicit formulas that are valid in the case of field~$\mathbb F_3$.

Assume that $C$ is an arbitrary symmetric matrix $n\times n$, whose elements belong to~$\mathbb F_3$, and 
${\mathbf y}^T C {\mathbf y}$ is a quadratic form with this matrix (${\mathbf y}$ is a vector column of the corresponding dimension). 
The following sum represents an analog of the multidimensional Gaussian integral:
$$
\Gau(C)=\sum_{{\mathbf y}\in \mathbb F_3^n} \exp( 2\pi i\, {\mathbf y}^T C {\mathbf y}/3).
$$
In the case $n=1$, we get the so-called 
quadratic Gaussian sum $g(c)=\sum_{y\in \mathbb F_3} \exp( 2\pi i\,c y^2/3)$.
By elementary calculations, we make sure that $g(0)=3$, otherwise $g(c)=\left(\frac{c}{3}\right)i\sqrt{3}$.
See~\cite{ai} for the historical background of the calculation of the quadratic Gaussian sum for an arbitrary field $\mathbb F_p$, where $p$ is a prime number, $p>2$; see~\cite[Theorem~5.15]{lidlNider} for the general case of field~$\mathbb F_q$, $q=p^d$.

\begin{remark}
\label{rem:cong}
If matrices $C$ and $A$ are congruent, i.e., $A=P^T C P$, where $P$ is a non-singular $n\times n$ matrix, then $\Gau(C)=\Gau(A)$.
\end{remark}

Remark~\ref{rem:cong} is valid because by putting ${\mathbf y}'= P {\mathbf y}$ we reduce the sum $\Gau(A)$ to $\Gau(C)$.

\begin{lemma}[a particular case of Lemma~8 in~\cite{EJC} for the field $\mathbb F_3$]
\label{lem:gauss}
Consider a symmetric $n\times n$ matrix $C$ of rank~$r$, whose elements belong to $\mathbb F_3$; let $\det C_r$ be an arbitrary non-zero principal minor of order $r$ of the matrix $C$. 
The following formula is valid:
\begin{equation}
\label{eq:gauss}
\frac{\Gau(C)}{3^n} =\left( \frac{\det C_r}3\right) \left[\frac{i}{\sqrt{3}} \right]^r.
\end{equation}
\end{lemma}

In view of Remark~\ref{rem:cong} the proof of Lemma~\ref{lem:gauss} is reduced to considering the diagonal case
(see~\cite[Chapters IV]{serre} for the reduction of a quadratic form over a finite field to the diagonal form). 
By factorization, we reduce the diagonal case to the one-dimensional variant considered above.
The multiplicative property of the Legendre symbol allows us to write the final result in the form~\eqref{eq:gauss}.

\begin{conclusion}
\label{concl2}
For any symmetric matrix~$C$ of rank~$r$, the value $\left( \frac{\det C_r}3\right)$ is independent of the choice of $C_r$. 
\end{conclusion}

\begin{conclusion}
\label{concl3}
Assume that all elements of a symmetric matrix $C$ are linear functions of a certain set of variables $\alpha\in\mathbb (F_3^*)^k$, while $r(C(\alpha))$ is the rank of this matrix. 
Then
$$
\sum_{\substack{
\alpha:\ \alpha\in (F_3^*)^k, \\
r(C(\alpha))\bmod 2 = 1}} \Gau(C(\alpha)) =0.
$$
\end{conclusion}

\begin{proof}
Let us replace $\alpha$ in the sum under consideration by $-\alpha$. 
Note that when calculating $\Gau(C(-\alpha))$ we replace the sign of the value $\det C_r$ with the opposite (i.e., use the term $(-1)^r \det C_r$). 
Therefore, the considered sum equals itself with the opposite sign.
\end{proof}

Let now $G$ be an arbitrary multigraph and let $L(G;\lambda)$ be a weighted Laplace--Kirchhoff matrix of graph~$G$, i.e., $L(G;\lambda)=B \Lambda B^T$,
where $B$ is the oriented incidence matrix, while $\Lambda$ is the diagonal matrix, whose diagonal elements are equal to $\lambda_e$.

\begin{lemma}[cf. Theorem~6 in~\cite{EJC}]
\label{lem:Laplas}
Let $C=L(G;{\mathbf x})$.  Then the right-hand side of formula~\eqref{eq:gauss} coincides with the right-hand side of formula~\eqref{Gau3}.
\end{lemma}
\begin{proof}
The principal minor of the matrix $L(G;{\mathbf x})$ equals the determinant of the submatrix obtained from $L(G;{\mathbf x})$ by deleting rows and columns, whose indices belong to the set $W$. According to Lemma~\ref{lem:gauss}, it suffices to prove that this minor coincides with $s(G/W;{\mathbf x})$.
This fact is well known in the case where $W=\emptyset$~(see, for example, \cite{MatricesAndGraphs}, as well as \cite{stanleyArticle} and references therein). The general case is reduced to a particular one,
because the matrix $L(G/W;{\mathbf x})$ is obtained from $L(G;{\mathbf x})$ by deleting the rows and columns corresponding to the vertices in $W$, and adding a row and column corresponding to the resulting contracted vertex. These new elements of the matrix $L(G/W;{\mathbf x})$ are fully determined by the remaining part of the matrix because the sum of the elements of any row or column in this matrix equals zero.
\end{proof}

\section{The Heawood theorem and the Fourier transform}

To prove Theorem~\ref{th:main} we need the Heawood representation for $\Tait_0(G)$ as the number of nowhere-zero solutions of a system of linear equations over~$\mathbb F_3$.  

\begin{proposition}[\cite{heawood}]
\label{th:heawood}
Let $G$ be a maximal planar graph. Let us associate each face $F$ of graph~$G$ with variable $\sigma(F)$, 
which takes on values in set $\mathbb{F}_3^*$.
Then $\Tait_0(G)$ equals the number of all possible sets of spins $(\sigma(F), F\in \F(G))$,
such that for any graph vertex~$v$, the sum $\sigma(F)$ calculated with respect to all faces~$F$ that contain~$v$, equals zero.
\end{proposition}

See also~\cite[Theorem 9.3.4]{Ore} for the proof of this proposition; see~\cite{belaga} (as well as~\cite{arXiv1}) for its other proofs.

We also make use of some simple properties of the Fourier transform over field~${\mathbb F}_3$. 
Consider complex-valued functions $f(k)$, whose argument $k$ belongs to field $\mathbb{F}_3$.
The inverse Fourier transform of such functions is usually understood as the function $\widehat f(y)=\sum\limits_{k\in \mathbb{F}_3} f(k)\frac{\exp(2\pi i\, ky/3)}{3}$.

Let ${\mathbf 1}(k)$ be the function $f(k)$ that is identically equal to one; let symbol $\delta(y)$ denote the delta function (the Kronecker symbol):
$\delta(0)=1$, $\delta(y)=0$ with all $y\in\mathbb F_3^*$. We can easily make sure that 
\begin{equation}
\label{widehat1}
\widehat{{\mathbf 1}}(y)=\delta(y). 
\end{equation}
Relation~\eqref{widehat1} implies the following formula, which plays an important role in further calculations:
\begin{equation}
\label{alpha}
\sum_{y\in \mathbb{F}_3^*} \exp(2\pi i\, ky/3) =3\delta(k)-1=3\delta(k^2)-1=\sum_{a\in \mathbb{F}_3^*} \exp(2\pi i\, k^2 a/3).
\end{equation}

\begin{remark}
\label{rem:alpha}
In the case of a finite field, we can consider the function $f(k)=1-\delta(k)$ as the norm of an element of the finite field. Thus, the sum $\sum_{y\in \mathbb{F}_p^*} \exp(2\pi i\, ky/p)$ is the Fourier transform of the norm raised to a certain power.
An analog of this sum in the case of a real field is the Fourier transform of a (generalized) function $|k|^\gamma$.
In quantum field theory, such functions often represent the so-called propagators of Feynman amplitudes. 
The paper~\cite[p.~691]{syman} has given rise to the parametric representation of the integrals of propagators of Feynman amplitudes as integrals of the exponent.
In the mentioned paper, K.~Symanzik uses the symbol $\alpha$ for the analog of variable~$a$ introduced by us.
The technique based on the use of this representation in quantum field theory (in the next section we implement its simplified analog for~$\mathbb F_3$) is called the
$\alpha$-representation. Following this tradition, we will use the same notation.
\end{remark}

\section{The $\alpha$-representation for $\Tait_0(G)$}

\begin{proof}[Proof of theorem~\ref{th:main}]
According to the Heawood theorem (Proposition~\ref{th:heawood}), 
\begin{equation}
\label{eq:S}
\Tait_0(G)=\sum_{\sigma\in\{-1,1\}^{\F(G)}}\ \prod_{v\in V(G)} \delta(\sum_{F:v\in F} \sigma(F)).
\end{equation}

Let us modify the right-hand side of formula~\eqref{eq:S}, using the fact that each $\delta$-function represents the inverse Fourier transform of the function ${\mathbf 1}(\cdot)$ (see~\eqref{widehat1}). 
Representing the product of exponents as the exponent of the sum and changing the summation order, we conclude that
$$
\Tait_0(G)=\sum_{{\mathbf k}\in {\mathbb F}_3^{V(G)}}\ \sum_{\sigma\in\{-1,1\}^{\F(G)}} \exp\left(\frac{2\pi i}3 \, \sum_{v\in V(G)} k_v\sum_{F: v\in F} \sigma(F) \right)/
3^{|V(G)|}.
$$
We can represent the sum in the exponent in another way, namely, 
$$
\sum_{v\in V(G)} k_v\sum_{F: v\in F} \sigma(F) =\sum_{F\in \F(G)} \sigma(F) \sum_{v\in F} k_v.
$$
This allows us to use the formula~\eqref{alpha}. We obtain the relation
$$
\Tait_0(G)=\sum_{{\mathbf k}\in {\mathbb F}_3^{V(G)}}\ \sum_{\alpha\in\{-1,1\}^{\F(G)}} \exp\left(\frac{2\pi i}3 \, \sum_{F\in \F(G)} \alpha(F) \left( \sum_{v\in F} k_v\right)^2\ 
\right)/3^{|V(G)|}.
$$
The sum in the exponent can be expressed differently, specifically,
$$\sum_{F\in \F(G)} \alpha(F) \left( \sum_{v\in F} k_v\right)^2=\sum_{v_1\in V}\sum_{v_2 \in V}k_{v_1}k_{v_2}\sum_{\substack{F \in \F(G):\\v_1 \in F, v_2 \in F}}\alpha(F).$$
Taking into account the fact that $2(\alpha(F'_e)+\alpha(F''_e))=-(\alpha(F'_e)+\alpha(F''_e))$, we get the relation
$$
\Tait_0(G)=\sum_{\alpha\in\{-1,1\}^{\F(G)}}\, \sum_{\mathbf k\in {\mathbb F}_3^{V(G)}} \frac{ \exp( 2\pi i\,{\mathbf k}\, L({\mathbf x}(-\alpha))\,{\mathbf k}^T\ /3)}{3^{|V(G)|}}
=\sum_{\alpha\in\{-1,1\}^{\F(G)}} \frac{\Gau(L({\mathbf x}(-\alpha)))}{3^{|V(G)|}}
$$
(here $\mathbf k=(k_v, v\in V(G))$). Note that we iterate over all possible values of $\alpha$, thus we can replace ${\mathbf x}(-\alpha)$ with ${\mathbf x}(\alpha)$. In view of Corollary~\ref{concl3} and Lemma~\ref{lem:Laplas} this equality is equivalent to the assertion of Theorem~\ref{th:main}.
\end{proof}  

\section{Conclusion}
In the case of finite fields, the base of the $\alpha$-representation is the interpretation of the desired value as the number of nowhere-zero solutions to a system of linear equations. For the number of Tait colorings, this base is ensured by the Heawood theorem. 
The application of formula~\eqref{alpha} allows us to reduce further calculations to the evaluation of multidimensional Gaussian sums, for which the principal minors of the matrix of the quadratic form can be expressed explicitly in terms of the Legendre symbol. In the case of Tait colorings, we can visually interpret these minors as the sum with respect to spanning trees. 
We plan to further develop the $\alpha$-representation technique in this evident case, which is also related to the assertion of the Four Color Theorem.

\end{document}